\documentclass[11pt]{article}
\usepackage[T1]{fontenc}
\usepackage[utf8]{inputenc}
\usepackage{lmodern}
\usepackage{amsmath,amssymb,amsthm,mathrsfs}
\usepackage{mathtools}
\usepackage[margin=1in]{geometry}
\usepackage{microtype}
\usepackage{parskip}
\usepackage{needspace}
\usepackage{enumitem}
\usepackage[hidelinks,bookmarksnumbered=true]{hyperref}

\allowdisplaybreaks[2]
\newtheorem{theorem}{Theorem}[section]
\newtheorem{conjecture}[theorem]{Conjecture}
\newtheorem{lemma}[theorem]{Lemma}
\newtheorem{proposition}[theorem]{Proposition}
\theoremstyle{remark}
\newtheorem{remark}[theorem]{Remark}
\newcommand{\F}{\mathbb F}
\newcommand{\PP}{\mathbb P}
\newcommand{\E}{\mathcal E}
\newcommand{\1}{\mathbf1}
\newcommand{\norm}[1]{\left\lVert#1\right\rVert}
\title{Sharp spherical $L^p\to L^4$ extension estimates\\over finite fields in odd dimensions}
\author{Thang Pham \thanks{Institute of Mathematics and Interdisciplinary Sciences at Xidian University, China. \newline
\hspace*{0.5cm} Email: phamanhthang.vnu@gmail.com} \and Boqing Xue\thanks{Institute of Mathematical Sciences, ShanghaiTech University. Email: xuebq@shanghaitech.edu.cn}}
\date{}
\begin{document}
\maketitle
\begin{abstract}
Let $q$ be an odd prime power, let $d\ge3$ be odd, and let $S_j=\{\xi\in\mathbb F_q^d:\xi_1^2+\cdots+\xi_d^2=j\}$ with $j\in \mathbb F_q^*$. We prove the sharp $L^{4(d+1)/(3d+1)}(S_j,d\sigma)\to L^4(\mathbb F_q^d,dc)$ extension estimate whenever $\eta_q((-1)^{(d-1)/2}j)=-1$, where $\eta_q$ is the quadratic character of $\mathbb F_q$. This confirms the odd-dimensional spherical extension conjecture of Koh, Pham, and Vinh (J. Funct. Anal. 2021) over every finite field of odd characteristic. Together with the known estimate in the complementary sign class, this determines the sharp $L^p \to L^4$ extension range for all nonzero-radius spheres in odd dimensions. The key ingredient in our proof is a new rectangle estimate, proved by using a weighted point-hyperplane incidence bound. We also apply the endpoint estimate to obtain sufficient conditions for the unique recovery of sparse signals from incomplete Fourier data.
\end{abstract}

{\small
}

\tableofcontents

\section{Introduction}\label{sec:introduction}

Fourier restriction over finite fields asks {for norm estimates for the Fourier transform restricted to an algebraic surface}. Its dual formulation concerns the integrability of Fourier extensions of functions supported on the surface. Since the work of Mockenhaupt and Tao~\cite{MT}, this problem has been studied for paraboloids, cones, spheres and other quadratic varieties. Spheres are of particular interest because their restriction estimates are closely related to questions in discrete geometry, additive combinatorics and analytic number theory.

Let $q$ be an odd prime power, and fix a nontrivial additive character $\chi:\F_q\to\mathbb C^\times$. For $x,\xi\in\F_q^d$, write
\[
 x\cdot\xi=\sum_{\nu=1}^d x_\nu\xi_\nu,
 \qquad Q(\xi)=\xi\cdot\xi.
\]
For $j\in\F_q^*$, define the sphere
\[
 S_j=\{\xi\in\F_q^d:Q(\xi)=j\}.
\]
We follow the Fourier notation of~\cite{KPV}. Let $dc$ denote counting measure on $\F_q^d$ and let $d\sigma$ denote normalized surface measure on $S_j$. For $1\le p,r<\infty$, set
\[
 \norm{f}_{L^p(S_j,d\sigma)}
 =\left(\frac1{|S_j|}\sum_{\xi\in S_j}|f(\xi)|^p\right)^{1/p},
 \qquad
 \norm{g}_{L^r(\F_q^d,dc)}
 =\left(\sum_{x\in\F_q^d}|g(x)|^r\right)^{1/r},
\]
with the usual supremum norms when an exponent is infinite. For $A\subseteq\F_q^d$, the notation $L^r(A,dc)$ also refers to counting measure. The extension of $f:S_j\to\mathbb C$ is
\[
 (f\,d\sigma)^\vee(x)
 =\frac1{|S_j|}\sum_{\xi\in S_j}f(\xi)\chi(x\cdot\xi),\qquad (x\in \F_q^d).
\]
We denote by $R^*_{S_j}(p\to r)$ the least constant in
\begin{equation}\label{eq:extension}
 \norm{(f\,d\sigma)^\vee}_{L^r(\F_q^d,dc)}  \le R^*_{S_j}(p\to r)\norm{f}_{L^p(S_j,d\sigma)}.
\end{equation}
The problem is to identify the exponents for which this constant is bounded independently of $q$ and $j$.

For $g:\F_q^d\to\mathbb C$, define the normalized and unnormalized Fourier transforms by
\[
 \widehat g(\xi)=q^{-d}\sum_{x\in\F_q^d}g(x)\chi(-x\cdot\xi),  \qquad
 \widetilde g(\xi)=\sum_{x\in\F_q^d}g(x)\chi(-x\cdot\xi) =q^d\widehat g(\xi).
\]
For a function $f$ on $\F_q^d$, set
\[
 f^\vee(x)=q^{-d}\sum_{\xi\in\F_q^d}f(\xi)\chi(x\cdot\xi).
\]
Fourier inversion and Plancherel give
\[
 g(x)=(\widetilde g)^\vee(x) =\sum_{\xi\in\F_q^d}\widehat g(\xi)\chi(x\cdot\xi),
 \qquad
 \norm{\widetilde g}_{L^2(\F_q^d,dc)} =q^{d/2}\norm{g}_{L^2(\F_q^d,dc)}.
\]
By duality,~\eqref{eq:extension} is equivalent to
\[
 \norm{\widetilde g}_{L^{p'}(S_j,d\sigma)} \le R^*_{S_j}(p\to r)\norm{g}_{L^{r'}(\F_q^d,dc)},
\]
where primes denote H\"older conjugates.

Throughout the paper, $\eta_q$ denotes the quadratic character of $\F_q$, extended by $\eta_q(0)=0$, and $|A|$ denotes the cardinality of a finite set $A$. We write ${X\lesssim_d Y}$ if $X\le C_dY$ for a constant depending only on $d$, and ${X\approx_d Y}$ if both inequalities hold. In particular, these constants are independent of the field, its characteristic and the radius.

\subsection{Earlier work}

Iosevich and Koh~\cite{IK08} established extension estimates for nondegenerate quadratic surfaces over finite fields. For a nonzero-radius sphere, their Fourier decay estimates give the finite-field Stein--Tomas bound
\begin{equation}\label{eq:stein-tomas}
 R^*_{S_j}\left(2\to\frac{2(d+1)}{d-1}\right)\lesssim_d1.
\end{equation}
Their subsequent work~\cite{IK10} improved the available {${L^p\to L^4}$} estimates in even dimensions and examined the obstructions arising from affine subspaces contained in a sphere. {Two questions have received particular attention:} the extension problem with source exponent two, which has direct applications to distance sets, and the problem with target exponent four, which can be studied through additive energy.

Several works have also considered restriction estimates for special classes of functions. Kang and Koh~\cite{KK14} proved the conjectured $L^2$ restriction estimate for nonzero-radius spheres in even dimensions when the ambient test functions are homogeneous of degree zero. Here this means that $g({\lambda}x)=g(x)$ for every ${\lambda}\in\F_q^*$ and $x\in\F_q^d$. Kang and Koh~\cite{KK26} subsequently introduced a connection with $S$-operators and obtained the corresponding estimates in {odd dimensions when $\eta_q((-1)^{(d-1)/2}j)=-1$}. The construction lifts the sphere to the homogeneous quadratic surface 
\[
 \{(x,{\lambda})\in\F_q^{d+1}:Q(x)=j{\lambda}^2\}.
\]
Kang, Koh and Yang~\cite{KKY26} determined the mapping properties of the $S$-operator. In particular, their Corollary~1.9 identifies a limitation of proving the even-dimensional conjecture for arbitrary functions by a two-step lifting argument with one fixed intermediate Lebesgue exponent. More recently, Pham and Xue~\cite{PX26} established the sharp estimate
\[
 R^*_{S_j}\left(2\to\frac{2d+4}{d}\right)\lesssim_d1  \qquad(d\ge4\text{ even})
\]
for nonzero level sets of arbitrary nondegenerate quadratic forms over odd finite fields.

For target exponent four, Iosevich, Koh, Lee, Pham and Shen~\cite[Theorem~1.5]{IKLPS} proved that, for every even $d\ge4$ and every $j\ne0$,
\[
 R^*_{S_j}(p\to4)\lesssim_d1
 \quad\Longleftrightarrow\quad
 p\ge\frac{4d}{3d-2}.
\]
That paper also settled the $L^2$ extension problem for the zero-radius sphere.

Suppose henceforth that $d=2h+1\ge3$. The {dimensions of affine subspaces in $S_j$ depend on} $\eta_q((-1)^h j)$. If this sign is positive, the sphere contains an affine $h$-dimensional subspace. The corresponding obstruction, together with interpolation of \eqref{eq:stein-tomas} and the trivial ${L^1\to L^\infty}$ estimate, gives the sharp range
\begin{equation}\label{eq:positive-sign}
 R^*_{S_j}(p\to4)\lesssim_d1  \quad\Longleftrightarrow\quad p\ge\frac{4(d-1)}{3d-5}.
\end{equation}
See~\cite{IK10,KPV} for the affine-subspace obstruction. If $\eta_q((-1)^h j)=-1$, the largest affine subspaces of $S_j$ have dimension $h-1$. In this case, the same test requires only 
\[
 p\ge p_d,\qquad p_d=\frac{4(d+1)}{3d+1},
\]
leaving a gap between the necessary condition and the known extension estimates.

Koh, Pham and Vinh~\cite[Theorems~1.8--1.9]{KPV} obtained 
\begin{equation}\label{eq:previous}
 R^*_{S_j}\left(\frac{4d}{3d-2}\to4\right)\lesssim_d 1
\end{equation}
for primitive radii when $d\equiv1\pmod4$, and when $d\equiv3\pmod4$ and $q\equiv1\pmod4$. Dilation extends these results to all nonsquare radii in the same cases. Indeed, the quotient of two nonsquares is a square, and multiplication by a scalar identifies the corresponding extension problems. Their proof uses association schemes and additive energy. Koh and Pham~\cite[Theorem~1.3]{KP} proved \eqref{eq:previous} when $d\equiv3\pmod4$, $q\equiv3\pmod4$, and $j$ is a nonzero square, using point--hyperplane incidences derived from cone restriction. Thus, the estimate \eqref{eq:previous} was available throughout the negative sign class.

\subsection{The conjecture and the main result}

Koh, Pham and Vinh~\cite[Conjecture~2.3]{KPV} proposed the following sharp endpoint. It is restated by Koh and Pham~\cite[Conjecture~1.2]{KP}. We use the quadratic character to combine its three cases into one statement.

\begin{conjecture}\label{conj:main}
Let $q$ be an odd prime power, let $d=2h+1\ge3$, and let $j\in\F_q^*$ satisfy $\eta_q((-1)^h j)=-1$. Then
\[
 R^*_{S_j}\left(\frac{4(d+1)}{3d+1}\to4\right)\lesssim_d1.
\]
\end{conjecture}

The sign condition in Conjecture~\ref{conj:main} is equivalent to the following three cases:
\[
\begin{array}{c|c|c}
 d\pmod4 & q\pmod4 & j\\ \hline
 1 & 1\text{ or }3 & \text{nonsquare}\\
 3 & 1 & \text{nonsquare}\\
 3 & 3 & \text{nonzero square}.
\end{array}
\]
In this paper, we will prove this conjecture.

\begin{theorem}\label{thm:main}
Let $q$ be an odd prime power, let $d=2h+1\ge3$, and suppose that $j\in\F_q^*$ satisfies $\eta_q((-1)^h j)=-1$. Then
\begin{equation}\label{eq:main}
 \norm{(f\,d\sigma)^\vee}_{L^4(\F_q^d,dc)}
 \le C_d\norm{f}_{L^{p_d}(S_j,d\sigma)},
 \qquad p_d=\frac{4(d+1)}{3d+1},
\end{equation}
for every complex-valued function $f$ on $S_j$. The constant $C_d$ depends only on $d$. Moreover, for ${1\le p\le\infty}$, a uniform ${L^p\to L^4}$ extension estimate in this sign class holds if and only if ${p\ge p_d}$.
\end{theorem}

The {theorem lowers the source exponent in}~\eqref{eq:previous} {by}
\[
 \frac{4d}{3d-2}-p_d
 =\frac{8}{(3d-2)(3d+1)}.
\]
For example, the sharp source exponents in dimensions three, five and seven are $8/5$, $3/2$ and $16/11$, respectively. Together with~\eqref{eq:positive-sign}, the theorem determines the {sharp ${L^p\to L^4}$} extension range for every nonzero-radius sphere in odd dimensions.

The proof uses a mixed additive energy estimate. For $A,B\subseteq\F_q^d$, define
\[
 \E(A,B)  =|\{(a,b,a',b')\in A\times B\times A\times B: a+b=a'+b'\}|.
\]
Set $\tau=q^{h-1}=q^{(d-3)/2}$. We prove that, for all $A,B\subseteq S_j$,
\begin{equation}\label{eq:mixed-intro}
 \E(A,B)\lesssim_d  \tau^{2/(d+1)}|A|^{(2d+1)/(d+1)}|B|^{d/(d+1)} +\tau|A|^2+\tau|A||B|.
\end{equation}
A bilinear summation of dyadic pieces then gives the strong endpoint in~\eqref{eq:main} without a logarithmic loss.

We compare our argument with the method of~\cite{KPV}. Koh, Pham and Vinh~\cite{KPV} combine their energy estimate
\[
 \E(A,A)\lesssim_d \frac{|A|^3}{q} +q^{(d-2)/2}|A|^2
\]
with the trivial bound $\E(A,A)\le |A|^3$ and Fourier orthogonality to obtain~\eqref{eq:previous}. The source exponent $4d/(3d-2)$ is determined by the transition at $|A|\approx q^{(d-2)/2}$, where the cubic bound and the quadratic term coincide.

Our weighted incidence argument gives
\[
 \E(A,A)\lesssim_d \tau^{2/(d+1)}|A|^{4/p_d}  \qquad (|A|\ge\tau).
\]
Together with the cubic bound for $|A|\le\tau$, this yields the sharp $L^{p_d}\to L^4$ estimate for characteristic functions. The essential improvement occurs in the range
\[
 \tau<|A|\le q^{(d-2)/2},
\]
where the preceding Koh--Pham--Vinh estimate, combined with the trivial estimate, gives only $O_d(|A|^3)$. To obtain the endpoint for arbitrary functions, we use the full mixed estimate~\eqref{eq:mixed-intro} together with $\E(A,B)\le |A|^2|B|$. Its unequal cardinality exponents give summable decay in the separation of
dyadic indices, yielding~\eqref{eq:main} without a logarithmic loss.

The proof of Theorem~\ref{thm:main} does not require sharp additive energy estimates. We do not claim that the estimates proved here or those of~\cite{KPV} are sharp in general. Determining the optimal energy bounds remains an interesting independent problem.

In Section~\ref{sec:recovery}, we apply Theorem~\ref{thm:main} to the unique recovery of sparse signals from incomplete Fourier data.

We remark that the $L^p\to L^4$ extension problem has also been studied for other varieties, for example, paraboloids~\cite{IKP10,LL12,GOS24}, cones~\cite{MT,GI25}, homogeneous varieties~\cite{KS12H}, algebraic curves~\cite{KS12C}, and quadratic surfaces~\cite{Ronda26}.

\section{Preliminaries}\label{sec:preliminaries}

\subsection{Weighted projective incidences}

We begin with a Cauchy--Schwarz incidence estimate over an arbitrary field $\mathbb K$. For sets of points $P$ and lines $\mathcal L$, write 
\[
 I(P,\mathcal L)=|\{(x,L)\in P\times\mathcal L:x\in L\}|.
\]
A complete incidence configuration consists of a set of points and a set of hyperplanes such that every point lies on every hyperplane.

\begin{lemma}\label{lem:planar}
For finite sets $P$ of distinct points and $\mathcal L$ of distinct lines in $\PP^2(\mathbb K)$, we have
\[
 I(P,\mathcal L)\le {(|P||\mathcal L|)^{3/4}+|P|+|\mathcal L|}.
\]
\end{lemma}
\begin{proof}
The assertion is immediate if ${|P||\mathcal L|}=0$. Two distinct points determine a unique projective line, so
\[
 \sum_{L\in\mathcal L}{|P\cap L|(|P\cap L|-1)}  \le {|P|(|P|-1)}.
\]
Cauchy--Schwarz gives
\[
 {I(P,\mathcal L)^2  \le|\mathcal L|\sum_{L\in\mathcal L}|P\cap L|^2  \le|\mathcal L|\bigl(|P|^2+I(P,\mathcal L)\bigr)}.
\]
Hence, 
\[
 {I(P,\mathcal L)\le |P|\sqrt{|\mathcal L|}+|\mathcal L|}.
\]
By projective duality, also
\[
 {I(P,\mathcal L)\le |\mathcal L|\sqrt{|P|}+|P|}.
\]
If ${|P|\le|\mathcal L|}$, use the first bound and ${|P|\sqrt{|\mathcal L|}\le(|P||\mathcal L|)^{3/4}}$. If ${|\mathcal L|\le|P|}$, use the second bound. This completes the proof.
\end{proof}

We next extend this estimate to weighted point--hyperplane incidences in higher dimensions.

\begin{proposition}\label{prop:weighted}
Let $D\ge2$. Give finite sets $P$ of points and $\mathcal H$ of distinct hyperplanes in $\PP^D(\mathbb K)$ positive {real} weights $w_x$ and $v_H$. Set
\[
 m=\sum_{x\in P}w_x,\qquad n=\sum_{H\in\mathcal H}v_H,
 \qquad I_w=\sum_{H\in\mathcal H}\sum_{x\in P\cap H}w_xv_H.
\]
Suppose that no complete incidence configuration $P_0\times\mathcal H_0$ has both $\sum_{x\in P_0}w_x\ge {\kappa}$ and $\sum_{H\in\mathcal H_0}v_H\ge {\kappa}$, where $\kappa>0$. Then
\begin{equation}\label{eq:weighted}
 I_w\lesssim_D {\kappa}^{2/(D+2)}(mn)^{(D+1)/(D+2)}+{\kappa}(m+n).
\end{equation}
\end{proposition}

\begin{proof}
We argue by induction on $D$. For the induction step, we project from each point and compare the resulting incidence counts with a weighted second moment.

Remove every vertex of weight at least ${\kappa}$. Its neighbors have total weight less than ${\kappa}$, by the hypothesis. The incidences removed have total weight at most ${\kappa}(m+n)$. If no incidences remain, the result follows. It therefore suffices prove the estimate when every weight is less than $\kappa$. The desired right-hand side is increasing in both masses.

For $D=2$, let ${\alpha_1<\cdots<\alpha_{r_P}}$ and ${\beta_1<\cdots<\beta_{r_{\mathcal H}}}$ be the distinct point and hyperplane weights, respectively. Set $\alpha_0=\beta_0=0$ and
\[
 P_i=\{x:w_x\ge\alpha_i\},\qquad
 \mathcal H_k=\{H:v_H\ge\beta_k\},\qquad
 \Delta\alpha_i=\alpha_i-\alpha_{i-1},\qquad
 \Delta\beta_k=\beta_k-\beta_{k-1}.
\]
For each point and hyperplane, summing the corresponding increments recovers its weight. Thus,
\[
I_w=\sum_{i=1}^{r_P}\sum_{k=1}^{r_{\mathcal H}}\Delta\alpha_i\Delta\beta_k I(P_i,\mathcal H_k), \qquad m=\sum_{i=1}^{r_P}\Delta\alpha_i|P_i|,\qquad
 n=\sum_{k=1}^{r_{\mathcal H}}\Delta\beta_k|\mathcal H_k|.
\]
Applying H\"older's inequality, we obtain
\[
\sum_{i=1}^{r_P}\Delta\alpha_i|P_i|^{3/4}= \sum_{i=1}^{r_P}(\Delta\alpha_i)^{1/4} (\Delta\alpha_i|P_i|)^{3/4} \le \left(\sum_{i=1}^{r_P}\Delta\alpha_i\right)^{1/4}  \left(\sum_{i=1}^{r_P}\Delta\alpha_i|P_i|\right)^{3/4}  =\alpha_{r_P}^{1/4}m^{3/4}.
\]
The same argument for the hyperplane weights gives
\[
 \sum_{k=1}^{r_{\mathcal H}}\Delta\beta_k|\mathcal H_k|^{3/4}  \le  \left(\sum_{k=1}^{r_{\mathcal H}}\Delta\beta_k\right)^{1/4}  \left(\sum_{k=1}^{r_{\mathcal H}}  \Delta\beta_k|\mathcal H_k|\right)^{3/4}  =\beta_{r_{\mathcal H}}^{1/4}n^{3/4}.
\]
Multiply the estimate in Lemma~\ref{lem:planar} by $\Delta\alpha_i\Delta\beta_k$ and sum over $i$ and $k$. Since ${\alpha_{r_P},\beta_{r_{\mathcal H}}<\kappa}$, we obtain
\[
 I_w\le(\alpha_{r_P}\beta_{r_{\mathcal H}})^{1/4}(mn)^{3/4} +\beta_{r_{\mathcal H}} m+\alpha_{r_P} n  \le\kappa^{1/2}(mn)^{3/4}+\kappa(m+n).
\]
This proves~\eqref{eq:weighted} in dimension two.

Suppose that $D\ge3$ and that the result holds in dimension $D-1$. By duality, assume $m\le n$. For each $x\in P$, let
\[
 d_x=\sum_{x\in H}v_H,
 \qquad
 T_x=\sum_{x\in H}v_H\sum_{y\in(P\setminus\{x\})\cap H}w_y.
\]
For each line $L$ through $x$ containing another point of $P$, let $W_L=\sum_{y\in(P\setminus\{x\})\cap L}w_y$. Project all lines through $x$ containing another point of $P$ to points of $\PP^{D-1}(\mathbb K)$, with weights $W_L$. The hyperplanes through $x$ project to distinct hyperplanes with their original weights. The two total masses are at most $m$ and $d_x$. Every complete projected configuration lifts to a complete configuration of the original points on its lines and the original hyperplanes. The lifted point sets {exclude $x$ and} are disjoint, so their total weight is unchanged. The hypothesis is therefore preserved. Induction, with $\gamma=D/(D+1)$, gives
\[
 T_x\lesssim_D {\kappa}^{2/(D+1)}(m d_x)^\gamma+{\kappa}m+{\kappa}d_x.
\]
Multiplying by $w_x$ and sum over $x$. Since $0<\gamma<1$, H\"older's inequality gives
\begin{align*}
\sum_{x\in P}w_xd_x^\gamma =\sum_{x\in P}w_x^{1-\gamma}(w_xd_x)^\gamma \le \left(\sum_{x\in P}w_x\right)^{1-\gamma} \left(\sum_{x\in P}w_xd_x\right)^\gamma =m^{1-\gamma}I_w^\gamma,
\end{align*}
where we used $\sum_{x\in P}w_x=m$ and $\sum_{x\in P}w_xd_x=I_w$. On the other hand, we bound $\sum_xw_xT_x$ from below. For $H\in\mathcal H$, set $s_H=\sum_{x\in P\cap H}w_x$ so that $I_w=\sum_{H\in\mathcal H}v_Hs_H$. Exchanging the order of summation in the definition of $T_x$ gives
\[
\sum_{x\in P}w_xT_x  =\sum_{H\in\mathcal H}v_H \sum_{\substack{x,y\in P\cap H\\ x\ne y}}w_xw_y  =\sum_{H\in\mathcal H}v_H\left(s_H^2-\sum_{x\in P\cap H}w_x^2\right),
\]
since $s_H^2=\sum_{x,y\in P\cap H}w_xw_y$ includes the diagonal terms $x=y$. By the Cauchy--Schwarz inequality,
\[
 I_w^2  =\left(\sum_{H\in\mathcal H}\sqrt{v_H}\cdot\sqrt{v_H}\,s_H\right)^2  \le\left(\sum_{H\in\mathcal H}v_H\right) \left(\sum_{H\in\mathcal H}v_Hs_H^2\right)  =n\sum_{H\in\mathcal H}v_Hs_H^2.
\]
Since every weight is less than $\kappa$, we have $w_x^2<\kappa w_x$, and therefore
\[
 \sum_{H\in\mathcal H}v_H\sum_{x\in P\cap H}w_x^2  \le\kappa\sum_{H\in\mathcal H}v_Hs_H =\kappa I_w.
\]
Combining the last three formulae, we obtain
\[
 \sum_{x\in P}w_xT_x\ge\frac{I_w^2}{n}-\kappa I_w.
\]
Consequently,
\[
 \frac{I_w^2}{n}\lesssim_D  {\kappa}^{2/(D+1)}mI_w^\gamma+{\kappa}m^2+{\kappa}I_w.
\]
Solving according to the largest term on the right gives
\[
 I_w\lesssim_D {\kappa^{2/(D+2)}(mn)^{(D+1)/(D+2)}} +\sqrt{{\kappa}}\,m\sqrt n+{\kappa}n.
\]
If $mn\ge {\kappa}^2$, then by $m\le n$ we have
\[
\sqrt{{\kappa}}\,m\sqrt n  \le\sqrt{{\kappa}}(mn)^{3/4}  \le{\kappa^{2/(D+2)}(mn)^{(D+1)/(D+2)}},
\]
where the last inequality follows from 
\[
\frac{{\kappa^{2/(D+2)}(mn)^{(D+1)/(D+2)}}}{\sqrt{{\kappa}}(mn)^{3/4}}  =\left(\frac{mn}{{\kappa}^2}\right)^{(D-2)/(4(D+2))}\ge 1.
\]
If $mn<{\kappa}^2$, use $I_w\le mn\le{\kappa^{2/(D+2)}(mn)^{(D+1)/(D+2)}}$ instead. Adding back the removed incidences proves~\eqref{eq:weighted}. This completes the induction.
\end{proof}

\begin{remark}
The corresponding unweighted estimate, with an implicit constant depending on $\kappa$, follows from Milojevi\'c, Sudakov and Tomon~\cite[Theorem~1.2]{MST}. The elementary proof of Proposition~\ref{prop:weighted} allows weights and makes the dependence on $\kappa$ explicit. This dependence is essential in the proof of Theorem~\ref{thm:main}.
\end{remark}

\begin{remark}\label{rem:lewko-incidences}
The sharp $L^p\to L^4$ estimate in Theorem~\ref{thm:main} requires only the elementary planar bound in Lemma~\ref{lem:planar}. Using the stronger planar incidence bound of Lewko~\cite[Theorem~1]{Lewko26} (see also Miao and Xie~\cite[Theorem~1.1]{MX} for odd prime fields), a similar argument gives, under the hypotheses of Proposition~\ref{prop:weighted},
\[
 I_w\lesssim_D  \kappa^{2/(D+1)}(mn)^{D/(D+1)}  +\kappa(m+n)+\frac{mn}{\mathfrak p},
\]
where $\mathfrak p=\operatorname{char}\mathbb K$ in positive characteristic, and the last term is omitted in characteristic zero. In positive characteristic, the last term is absorbed by the first whenever $mn\le\kappa^2\mathfrak p^{D+1}$.
\end{remark}

\subsection{A geometric property of the sphere}

For the rest of the paper, let $d=2h+1\ge3$ and assume $\eta_q((-1)^h j)=-1$. Recall that $\tau=q^{h-1}=q^{(d-3)/2}$. The following lemma gives the geometric properties of $S_j$ needed for the weighted incidence argument and the sharpness example.
\begin{lemma}\label{lem:geometry}
The largest dimension of an affine subspace contained in $S_j$ is $h-1$. Every affine $h$-dimensional subspace $F$ satisfies $|F\cap S_j|\le2\tau$. Moreover, ${|S_j|=q^{d-1}-q^h\approx_d q^{d-1}}$.
\end{lemma}
\begin{proof}
We first determine the largest affine subspaces contained in $S_j$. The intersection bound will then follow from the polynomial zero bound, and the size of $S_j$ from a Gauss sum calculation.

If $z+U\subset S_j$, {then, for every $u\in U$,}
\[
 {Q(z\pm u)=j\pm2z\cdot u+Q(u)=j}.
\]
Since $q$ is odd, $z\cdot u=Q(u)=0$. Polarization gives $2u\cdot v=Q(u+v)-Q(u)-Q(v)=0$ for $u,v\in U$. Thus, $U\subset z^\perp$ is totally isotropic. The restriction of the dot product to $z^\perp$ is nondegenerate, since $Q(z)=j\ne0$. If $\dim U=h$, choose a basis $u_1,\ldots,u_h$ and vectors $v_1,\ldots,v_h\in z^\perp$ with $u_i\cdot v_{{\nu}}=\delta_{i{\nu}}$. The Gram matrix in the basis $z,u_1,\ldots,u_h,v_1,\ldots,v_h$ has determinant
\[
 j\det\begin{pmatrix}0&I_h\\I_h&{(v_\mu\cdot v_\nu)_{\mu,\nu=1}^h}\end{pmatrix}=(-1)^h j.
\]
Since the Gram matrix in the standard basis is the identity, this determinant is the square of a nonzero change-of-basis determinant. This contradicts the sign assumption. Thus, $S_j$ contains no affine subspace of dimension at least $h$.

The sphere is nonempty: the sets of squares and of $j$ minus a square each have $(q+1)/2$ elements and intersect. Choose $z\in S_j$. Every nondegenerate quadratic form of dimension at least three over $\F_q$ has a nonzero isotropic vector. Indeed, after diagonalization, the sets $\{a x^2:x\in\F_q\}$ and $\{-c-b y^2:y\in\F_q\}$ intersect for nonzero $a,b,c$. An isotropic vector $u$ can be paired with $v$ satisfying $u\cdot v=1$; replace $v$ by $v-Q(v)u/2$ to make it isotropic. Successively splitting off these hyperbolic planes in $z^\perp$ produces a totally isotropic space $U$ of dimension $h-1$. Thus, $z+U\subset S_j$. For $h=1$, take $U=\{0\}$.

The restriction of $Q-j$ to an affine $h$-flat is a nonzero polynomial of degree at most two. The elementary polynomial zero bound over $\F_q$ therefore gives $|F\cap S_j|\le2q^{h-1}$.

Finally, we recall the standard computation of $|S_j|$ using quadratic Gauss sums. Let $G=\sum_{u\in\F_q}\chi(u^2)$. The standard quadratic Gauss identities are
$G^2=\eta_q(-1)q$ and $\sum_u\chi({\lambda}u^2)=\eta_q({\lambda})G$ for ${\lambda}\ne0$. Character orthogonality now gives
\[
 |S_j|=q^{d-1}+q^{-1}G^d\sum_{{\lambda}\ne0}\eta_q({\lambda})\chi(-{\lambda}j)  =q^{d-1}+\eta_q((-1)^h j)q^h.
\]
This proves the asserted size formula and completes the proof.
\end{proof}

\subsection{Energy and Fourier orthogonality}

Recall the mixed additive energy $\E(A,B)$ defined in Section~\ref{sec:introduction}. It is symmetric in $A$ and $B$, and it satisfies the trivial bound
\begin{equation}\label{eq:trivial}
 \E(A,B)\le |A|^2|B|.
\end{equation}
Indeed, the three entries $a,z\in A$ and $b\in B$ determine $c=a+b-z$ uniquely. No ordering of $|A|$ and $|B|$ is needed.

The following consequence of character orthogonality allows us to pass from additive energy bounds to extension estimates.

\begin{lemma}\label{lem:orthogonality}
Suppose that $g_1,g_2:S_j\to\mathbb C$ are supported on $A,B\subset S_j$, respectively. Then
\begin{equation}\label{eq:bilinear}
 \norm{(g_1\,d\sigma)^\vee(g_2\,d\sigma)^\vee}_{L^2(\F_q^d,dc)}  \le\frac{q^{d/2}}{|S_j|^2}  \norm{g_1}_\infty\norm{g_2}_\infty\E(A,B)^{1/2}.
\end{equation}
Moreover,
\begin{equation}\label{eq:indicator-energy}
 \norm{(\1_A\,d\sigma)^\vee}_{L^4(\F_q^d,dc)}^4  =\frac{q^d}{|S_j|^4}\E(A,A).
\end{equation}

\end{lemma}
\begin{proof}
Expanding the squared norm and using character orthogonality gives
\[
 \norm{(g_1\,d\sigma)^\vee(g_2\,d\sigma)^\vee}_{{L^2(\F_q^d,dc)}}^2  =\frac{q^d}{|S_j|^4}  \sum_{\substack{a,z\in A,\ b,c\in B\\a+b=z+c}}  g_1(a)g_2(b)\overline{g_1(z)g_2(c)}.
\]
Taking absolute values proves~\eqref{eq:bilinear}. Setting $g_1=g_2=\1_A$ gives~\eqref{eq:indicator-energy}. This completes the proof.
\end{proof}

\section{Proof of the main result}\label{sec:proof}

\subsection{A mixed additive energy estimate}

We first prove the estimate presented in the introduction. {We then apply it to products of dyadic pieces of the extension.}

\begin{proposition}\label{prop:mixed}
For all $A,B\subset S_j$, with $\tau=q^{h-1}$, we have 
\begin{equation}\label{eq:mixed}
 \E(A,B)\lesssim_d  \tau^{2/(d+1)}|A|^{(2d+1)/(d+1)}|B|^{d/(d+1)}  +\tau|A|^2+\tau|A||B|.
\end{equation}
\end{proposition}
\begin{proof}
Fix $z\in A$. The identity
\[
 Q(a+b-z)-j=2(a-z)\cdot(b-z)
\]
shows that every energy quadruple satisfies $(a-z)\cdot(b-z)=0$. Dropping the condition $a+b-z\in B$, we obtain
\[
\E(A,B)\le\sum_{z\in A}  |\{(a,b)\in A\times B:(a-z)\cdot(b-z)=0\}|.
\]
Summed over all $z\in A$, triples with $a=z$ or $b=z$ contribute at most $|A||B|+|A|^2$.

For the remaining pairs, fix $z\in A$ and consider the projective space $\PP(\F_q^d)=\PP^{d-1}(\F_q)$. Let $P_z$ be the set of projective points $[a-z]$ with $a\in A\setminus\{z\}$, and give the point $[u]\in P_z$ the weight
\[
 w_{[u]}=|\{a\in A\setminus\{z\}:[a-z]=[u]\}|,
\]
the number of points of $A\setminus\{z\}$ on the line $z+\langle u\rangle$. Let $\mathcal H_z$ be the set of projective hyperplanes $\PP((b-z)^\perp)$ with $b\in B\setminus\{z\}$, and give the hyperplane $H$ the weight
\[
 v_H=|\{b\in B\setminus\{z\}:\PP((b-z)^\perp)=H\}|.
\]
Since the dot product is nondegenerate, distinct normal directions give distinct hyperplanes. The total masses satisfy
\[
 m_z=\sum_{[u]\in P_z}w_{[u]}\le|A|,  \qquad  n_z=\sum_{H\in\mathcal H_z}v_H\le|B|,
\]
and the weighted incidence count of this configuration is
\[
 I_z=|\{(a,b)\in(A\setminus\{z\})\times(B\setminus\{z\}): (a-z)\cdot(b-z)=0\}|.
\]

We verify the hypothesis of Proposition~\ref{prop:weighted} with $D=d-1$ and $\kappa=2\tau+1$. Let $P_0\times\mathcal H_0$ be a complete incidence configuration with $P_0\subseteq P_z$ and $\mathcal H_0\subseteq\mathcal H_z$. Let $A_0$ be the set of $a\in A\setminus\{z\}$ with $[a-z]\in P_0$, and let $B_0$ be the
set of $b\in B\setminus\{z\}$ with $\PP((b-z)^\perp)\in\mathcal H_0$. Since distinct projective points have disjoint sets of endpoints,
\[
 \sum_{[u]\in P_0}w_{[u]}=|A_0|,  \qquad  \sum_{H\in\mathcal H_0}v_H=|B_0|.
\]
Completeness means that $(a-z)\cdot(b-z)=0$ for all $a\in A_0$ and $b\in B_0$. Let $U$ and $V$ be the linear spans of $\{a-z:a\in A_0\}$ and $\{b-z:b\in B_0\}$. By bilinearity, $U\perp V$, so $V\subseteq U^\perp$. Nondegeneracy gives $\dim U^\perp=d-\dim U$, and hence
\[
 \dim U+\dim V\le d=2h+1.
\]
Therefore, $\dim U\le h$ or $\dim V\le h$. Suppose that $\dim U\le h$; the other case is identical with $B_0$ and $V$ in place of $A_0$ and $U$. Every $a\in A_0$ lies in the affine subspace $z+U$, which has dimension at most $h$. Choose an affine $h$-flat $F\supseteq z+U$. Since $A_0\subseteq S_j\cap F$, Lemma~\ref{lem:geometry} gives
\[
 \sum_{[u]\in P_0}w_{[u]}=|A_0|\le|F\cap S_j|\le2\tau<\kappa.
\]
Thus, no complete incidence configuration has both masses at least $\kappa$, as required.

Proposition~\ref{prop:weighted} with $D=d-1$, together with $m_z\le|A|$, $n_z\le|B|$, and $\kappa=2\tau+1\le3\tau$, yields
\[
 I_z\le C_d\bigl(\tau^{2/(d+1)}(|A||B|)^{d/(d+1)} +\tau(|A|+|B|)\bigr).
\]
Here we used that the right-hand side of~\eqref{eq:weighted} is increasing in both masses. Summing over the $|A|$ choices of $z$ and adding the at most $|A|^2+|A||B|$ triples with $a=z$ or $b=z$, we obtain
\[
 \E(A,B)\le C_d\tau^{2/(d+1)}|A|^{(2d+1)/(d+1)}|B|^{d/(d+1)} +(C_d\tau+1)\bigl(|A|^2+|A||B|\bigr).
\]
Since $\tau\ge1$, this proves~\eqref{eq:mixed}.
\end{proof}

\subsection{The endpoint estimate}

\begin{proof}[Proof of the upper bound in Theorem~\ref{thm:main}]
We expand the square of the extension into products of dyadic pieces. The trivial energy bound controls the lower indices, and Proposition~\ref{prop:mixed} gives a summable geometric factor for the remaining indices.

The assertion is immediate for $f=0$. Assume that $f$ is nonzero. By homogeneity, suppose ${\sum_{\xi\in S_j}|f(\xi)|^{p_d}=1}$. For integers $i\ge1$, let
\[
 A_i=\{\xi{\in S_j}:2^{-i}<|f(\xi)|\le2^{1-i}\},\qquad  f_i=f\1_{A_i},\qquad a_i=2^{-p_di}|A_i|.
\]
These sets partition the nonzero support of $f$ and satisfy 
\begin{equation}\label{eq:masses}
\sum_{i\ge1}a_i\le1,\qquad |A_i|\le2^{p_di},\qquad \norm{f_i}_\infty\le2^{1-i}.
\end{equation}
Expand ${((f\,d\sigma)^\vee)^2}$ and apply the triangle inequality in ${L^2(\F_q^d,dc)}$. Equation~\eqref{eq:bilinear} and symmetry give
\begin{equation}\label{eq:sum-target}
 |S_j|^2q^{-d/2}\norm{(f\,d\sigma)^\vee}_{{L^4(\F_q^d,dc)}}^2  \lesssim\sum_{1\le i\le k}2^{-i-k}\E(A_i,A_k)^{1/2}.
\end{equation}
We will show that the sum on the right is ${\lesssim_d\tau^{1/(d+1)}}$. Choose the integer ${i_0}\ge0$ by $2^{p_d{i_0}}\le\tau<2^{p_d({i_0}+1)}$. 

For $i\le {i_0}$, apply~\eqref{eq:trivial} and~\eqref{eq:masses}. Since ${p_d<2}$ and ${3p_d/2-2=4/(3d+1)>0}$,
\begin{align}
\sum_{i\le {i_0}}\sum_{k\ge i}2^{-i-k}\E(A_i,A_k)^{1/2}
 &\le\sum_{i\le {i_0}}2^{(p_d-1)i}\sum_{k\ge i}2^{(p_d/2-1)k}\notag\\
 &\lesssim_d2^{(3p_d/2-2){i_0}}  \le\tau^{3/2-2/p_d}=\tau^{1/(d+1)}. \label{eq:low}
\end{align}
The sum is empty if ${i_0}=0$.

For $i>{i_0}$, apply~\eqref{eq:mixed}. Set
\[
 \alpha=\frac{2d+1}{2(d+1)},\qquad \beta=\frac{d}{2(d+1)},\qquad \gamma_d=\frac{d+1}{3d+1}.
\]
Then ${p_d\alpha-1=\gamma_d}$, ${p_d\beta-1=-\gamma_d}$, and ${\alpha+\beta=2/p_d>1}$. The contribution of the first term in \eqref{eq:mixed} is at most
\begin{align}
 &\tau^{1/(d+1)}
 \sum_{i>{i_0}}\sum_{k\ge i}2^{-\gamma_d(k-i)}a_i^\alpha a_k^\beta \lesssim_d\tau^{1/(d+1)}.\label{eq:high-main}
\end{align}
To justify the last inequality explicitly, for each integer $\ell\ge0$, H\"older and~\eqref{eq:masses} give
\[
\sum_{i>{i_0}}a_i^\alpha a_{i+\ell}^\beta\le\left(\sum_{i>{i_0}}a_i^{\alpha+\beta}\right)^{\alpha/(\alpha+\beta)}\left(\sum_{i>{i_0}}a_{i+\ell}^{\alpha+\beta}\right)^{\beta/(\alpha+\beta)}
 \le1.
\]
Now sum the geometric series ${\sum_{\ell\ge0}2^{-\gamma_d\ell}}$. In particular, there is no logarithmic loss in~\eqref{eq:high-main}.

The remaining two terms in~\eqref{eq:mixed} give, respectively,
\begin{align}
 \sqrt\tau\sum_{i>{i_0}}\sum_{k\ge i}2^{-i-k}|A_i|
 &\lesssim\sqrt\tau\sum_{i>{i_0}}2^{(p_d-2)i}
 \lesssim_d\tau^{1/(d+1)},\label{eq:linear-one}\\
 \sqrt\tau\sum_{i>{i_0}}\sum_{k\ge i}2^{-i-k}(|A_i||A_k|)^{1/2}
 &\lesssim_d\sqrt\tau\sum_{i>{i_0}}2^{(p_d-2)i}
 \lesssim_d\tau^{1/(d+1)}.\label{eq:linear-two}
\end{align}
Here ${2^{p_d{i_0}}\approx_d\tau}$ and ${1/2+(p_d-2)/p_d=3/2-2/p_d=1/(d+1)}$.

Combining~\eqref{eq:sum-target}--\eqref{eq:linear-two} gives
\[
 \norm{(f\,d\sigma)^\vee}_{{L^4(\F_q^d,dc)}}^2
 \lesssim_d\frac{q^{d/2}\tau^{1/(d+1)}}{|S_j|^2}
 \approx_d |S_j|^{-2/p_d}
 =\norm{f}_{L^{p_d}(S_j,d\sigma)}^2.
\]
For the middle comparison, use ${|S_j|\approx_d q^{d-1}}$ and
\[
 \frac d2+\frac{d-3}{2(d+1)}
 =(d-1)\left(2-\frac2{p_d}\right).
\]
This proves the upper bound for every odd $q$, with a constant depending only on $d$.
\end{proof}

\subsection{Sharpness}

\begin{proof}[Proof of the sharpness assertion in Theorem~\ref{thm:main}]
Let $F_0\subset S_j$ be an affine $(h-1)$-flat given by Lemma~\ref{lem:geometry}. Its additive energy is $|F_0|^3=q^{3(h-1)}$. Hence,
\[
 \norm{(\1_{F_0}\,d\sigma)^\vee}_{{L^4(\F_q^d,dc)}}
 =\frac{q^{d/4}|F_0|^{3/4}}{|S_j|}
 \approx_d q^{-(3d+1)/8},
 \qquad
 \norm{\1_{F_0}}_{L^p({S_j,}d\sigma)}
 \approx_d q^{-(d+1)/(2p)}.
\]
A uniform bound requires ${p\ge4(d+1)/(3d+1)}$. Conversely, normalized counting measure is a probability measure, so the endpoint implies every larger source exponent. This completes the proof of Theorem~\ref{thm:main}.
\end{proof}

\section{An application to signal recovery}\label{sec:recovery}

\subsection{Recovery from extension estimates}\label{subsec:recovery-extension}

Let $g:\F_q^d\to\mathbb C$ be supported on $E\subseteq\F_q^d$. Suppose that the unnormalized Fourier coefficients ${\widetilde g(\xi)}$ are known for $\xi\notin\Sigma$, where $\Sigma\subseteq\F_q^d$ is the known set of unobserved frequencies. We seek to determine $g$ uniquely among signals with at at most $|E|$
nonzero entries, without knowing its support. The classical uncertainty principle states that every nonzero function $v:\F_q^d\to\mathbb C$ satisfies
\[
 |\operatorname{supp}v|\,|\operatorname{supp}\widetilde v|\ge q^d.
\]
If two distinct candidate signals $g_1$ and $g_2$ have the same observed Fourier coefficients, then $v=g_1-g_2$ is nonzero, has at most $2|E|$ nonzero entries, and satisfies $\operatorname{supp}\widetilde v\subseteq\Sigma$. Applying the uncertainty principle to this difference gives the sufficient condition
\begin{equation}\label{eq:recovery-classical}
 |E||\Sigma|<\frac{q^d}{2}.
\end{equation}
This is the recovery argument of Donoho and Stark~\cite{DS}. See also~\cite{IM} for its finite-group formulation. In related work, Burstein, Iosevich and Nathan~\cite{BIN26} study stable recovery from randomly missing samples using the Fourier ratio and $\ell^1$ minimization.

Restriction estimates yield stronger uncertainty inequalities when the missing frequencies have additional structure. This connection is developed in~\cite{IJM, IM}. Applying it to Theorem~\ref{thm:main} gives the following result.

\begin{theorem}\label{thm:recovery}
Let $q$ be an odd prime power, let $d=2h+1\ge3$, and let $j\in\F_q^*$ satisfy $\eta_q((-1)^h j)=-1$. There is a constant $\varepsilon_d>0$, depending only on $d$, with the following property. Let $b\in\F_q^d$, let $\varnothing\ne\Sigma\subseteq b+S_j$, and let $E\subseteq\F_q^d$. If
\begin{equation}\label{eq:recovery-condition}
 |E||\Sigma|^{(d-1)/(d+1)}  <\varepsilon_d q^{\,d-1+4/(d+1)},
\end{equation}
then every signal $g:\F_q^d\to\mathbb C$ supported on $E$ is uniquely determined among all signals with at most $|E|$ nonzero entries by its Fourier coefficients outside $\Sigma$.
\end{theorem}

\begin{proof}
Multiplication by $\chi(-b\cdot x)$ shifts the Fourier support by $-b$ and preserves the spatial support. Thus, we may assume that $b=0$. Let $v\ne0$ have Fourier support in $\Sigma$. Extend ${\widetilde v|_\Sigma}$ by zero to $S_j$. Fourier inversion, Theorem~\ref{thm:main}, H\"older's inequality, and Plancherel give
\begin{align*}
\norm{v}_{L^4(\F_q^d,dc)} \lesssim_d q^{-d}|S_j|^{1-1/p_d} \norm{\widetilde v}_{L^{p_d}(\Sigma,dc)} \le q^{-d/2}|S_j|^{1-1/p_d} |\Sigma|^{1/p_d-1/2}\norm{v}_{L^2(\F_q^d,dc)}.
\end{align*}
Since ${p_d=4(d+1)/(3d+1)}$ and ${|S_j|\approx_d q^{d-1}}$,
\[
\norm{v}_{L^4(\F_q^d,dc)}^4 \lesssim_d q^{-d+1-4/(d+1)}|\Sigma|^{(d-1)/(d+1)} \norm{v}_{L^2(\F_q^d,dc)}^4.
\]
On the other hand,
\[
 \norm{v}_{L^2(\F_q^d,dc)}^4  \le |\operatorname{supp}v|\norm{v}_{L^4(\F_q^d,dc)}^4.
\]
Consequently,
\begin{equation}\label{eq:recovery-uncertainty}
|\operatorname{supp}v| \gtrsim_d q^{d-1+4/(d+1)}|\Sigma|^{-(d-1)/(d+1)}.
\end{equation}
If two signals with at most $|E|$ nonzero entries have the same observed coefficients, their difference has Fourier support in $\Sigma$ and spatial support of size at most $2|E|$. For $\varepsilon_d$ sufficiently small, \eqref{eq:recovery-condition} and~\eqref{eq:recovery-uncertainty} force this difference to vanish. This completes the proof.
\end{proof}

The same proof using the earlier estimate~\eqref{eq:previous} in place of Theorem~\ref{thm:main} gives the sufficient condition 
\begin{equation}\label{eq:recovery-previous}
 |E||\Sigma|^{(d-2)/d}  <\widetilde\varepsilon_d q^{d-1+2/d}.
\end{equation}
Indeed, the the source exponent ${p=4d/(3d-2)}$ satisfies ${4/p-2=(d-2)/d}$, and the fourth power of the corresponding factor is bounded by a constant multiple of $q^{-d+1-2/d}|\Sigma|^{(d-2)/d}$.

The Stein--Tomas estimate~\eqref{eq:stein-tomas} gives a condition independent of $|\Sigma|$:
\begin{equation}\label{eq:recovery-stein-tomas}
 |E|<c_dq^{(d+1)/2}.
\end{equation}
To see this, let $r=2(d+1)/(d-1)$. For every $v$ with Fourier support in $\Sigma$, Fourier inversion,~\eqref{eq:stein-tomas}, and Plancherel give
\[
 \norm{v}_{L^r(\F_q^d,dc)}  \lesssim_d q^{-d/2}|S_j|^{1/2}\norm{v}_{L^2(\F_q^d,dc)}  \lesssim_d q^{-1/2}\norm{v}_{L^2(\F_q^d,dc)}.
\]
For $v\ne0$, H\"older's inequality and $1/2-1/r=1/(d+1)$ imply ${|\operatorname{supp}v|\gtrsim_d q^{(d+1)/2}}$. Applying this to the difference of two candidate signals proves \eqref{eq:recovery-stein-tomas}.

\subsection{Recovery from additive energy}\label{subsec:recovery-energy}

The use of additive energy in uncertainty principles and signal recovery is developed in Iosevich--Mayeli~\cite{IM} and Aldahleh et al.~\cite{AIIJMP}. We record a version with an explicit logarithmic loss and include its proof.

\begin{lemma}\label{lem:energy-recovery}
Let $\varnothing\ne\Sigma\subseteq\F_q^d$, let $E\subseteq\F_q^d$, and let $M\ge1$. Suppose that
\[
 \E(A,A)\le M|A|^2\qquad\text{for every }A\subseteq\Sigma.
\]
There is an absolute constant $\varepsilon>0$ such that, if
\[
 |E|<\frac{\varepsilon q^d}{M\log^2(2|\Sigma|)},
\]
then every signal supported on $E$ is uniquely determined among all signals with at most $|E|$ nonzero entries by its Fourier coefficients outside $\Sigma$.
\end{lemma}

\begin{proof}
Let $f:\Sigma\to\mathbb C$ be nonzero. By homogeneity, assume that $\norm{f}_{L^2(\Sigma,dc)}=1$. Set $f_\flat=f\1_{\{|f|\le |\Sigma|^{-1}\}}$, and, for $i\ge0$, let
\[
 A_i=\{\xi\in\Sigma:2^{-i-1}<|f(\xi)|\le2^{-i},\   |f(\xi)|>|\Sigma|^{-1}\}.
\]
There are $O(\log(2|\Sigma|))$ nonempty sets $A_i$. Orthogonality and the energy hypothesis imply
\[
 \left\lVert\sum_{\xi\in A_i}f(\xi)\chi(x\cdot\xi) \right\rVert_{L^4(\F_q^d,dc)}^4  \le q^d2^{-4i}\E(A_i,A_i)  \le q^dM2^{-4i}|A_i|^2.
\]
Since $\sum_i2^{-2i}|A_i|\le4$, Cauchy--Schwarz gives
\[
 \sum_i2^{-i}|A_i|^{1/2}\lesssim\sqrt{\log(2|\Sigma|)}.
\]
Also, $\sum_{\xi\in\Sigma}|f_\flat(\xi)|\le1$. The triangle inequality therefore yields
\begin{equation}\label{eq:energy-weighted-transfer}
\left\lVert\sum_{\xi\in\Sigma}f(\xi)\chi(x\cdot\xi)  \right\rVert_{L^4(\F_q^d,dc)}  \lesssim q^{d/4}M^{1/4}\sqrt{\log(2|\Sigma|)} \norm{f}_{L^2(\Sigma,dc)}.
\end{equation}
If $v\ne0$ has Fourier support in $\Sigma$, Fourier inversion, Plancherel, and~\eqref{eq:energy-weighted-transfer} give
\[
 \norm{v}_{L^4(\F_q^d,dc)}^4  \lesssim q^{-d}M\log^2(2|\Sigma|)\norm{v}_{L^2(\F_q^d,dc)}^4.
\]
H\"older's inequality now implies
\[
 |\operatorname{supp}v|\gtrsim\frac{q^d}{M\log^2(2|\Sigma|)}.
\]
For $\varepsilon$ sufficiently small, two signals with at most $|E|$ nonzero entries and the same observed coefficients cannot have a nonzero difference. This completes the proof.
\end{proof}

We also use a refinement when the energy bound has a positive power beyond $|A|^2$. Suppose that $0<\delta\le1$ and
\[
 \E(A,A)\le K|A|^{2+\delta} \qquad\text{for every }A\subseteq\Sigma.
\]
Then the same uniqueness conclusion holds under 
\begin{equation}\label{eq:energy-power-recovery}
 |E|<\frac{c_\delta q^d}{K|\Sigma|^\delta}.
\end{equation}
Indeed, let $f:\Sigma\to\mathbb C$ be nonzero. By homogeneity, assume that $\norm{f}_{L^2(\Sigma,dc)}=1$. For $k\ge0$, define
\begin{align*}
 L_k&=\{\xi\in\Sigma:2^{-k-1}|\Sigma|^{-1/2}<|f(\xi)|\le2^{-k}|\Sigma|^{-1/2}\},\\
 H_k&=\{\xi\in\Sigma:2^k|\Sigma|^{-1/2}<|f(\xi)|\le2^{k+1}|\Sigma|^{-1/2}\}.
\end{align*}
We have $|L_k|\le |\Sigma|$ and $|H_k|\le |\Sigma|2^{-2k}$. By orthogonality and the energy hypothesis, the corresponding $L^4$ norms are bounded by
$q^{d/4}K^{1/4}|\Sigma|^{\delta/4}2^{-k}$ and $2q^{d/4}K^{1/4}|\Sigma|^{\delta/4}2^{-k\delta/2}$, respectively. Summing the two geometric series gives
\[
\left\lVert\sum_{\xi\in\Sigma}f(\xi)\chi(x\cdot\xi) \right\rVert_{L^4(\F_q^d,dc)}  \lesssim_\delta q^{d/4}K^{1/4}|\Sigma|^{\delta/4} \norm{f}_{L^2(\Sigma,dc)}.
\]
The Fourier inversion and support argument in the preceding proof now gives~\eqref{eq:energy-power-recovery}.

The next proposition states the consequences of the energy bounds of Koh--Pham--Vinh~\cite{KPV}, Koh--Pham~\cite{KP} and Shkredov~\cite{Shkredov} with the same uniqueness conclusion as Theorem~\ref{thm:recovery}.

\begin{proposition}\label{prop:recovery-earlier-energy}
Let $q$ be an odd prime power, let $d=2h+1\ge5$, and let $j\in\F_q^*$ satisfy $\eta_q((-1)^h j)=-1$. Let $b\in\F_q^d$, let $\varnothing\ne\Sigma\subseteq b+S_j$, and let $E\subseteq\F_q^d$. Set
\[
\tau=q^{(d-3)/2},\qquad a_d=2^d-d-1,
\]
and write
\[
 B_{q,d}(|\Sigma|)=\frac{|\Sigma|}{q}+q^{(d-2)/2}+
 \begin{cases}
 q^{(d-3)/4}|\Sigma|^{1/2},&d\equiv3\pmod4,\ q\equiv3\pmod4,\\
 0,&\text{otherwise}.
 \end{cases}
\]
There is a constant $c_d>0$, depending only on $d$, such that, if 
\begin{equation}\label{eq:recovery-earlier-energy}
 |E|<c_d  \max\left\{\frac{q^d}{B_{q,d}(|\Sigma|)\log^2(2|\Sigma|)},\  \frac{q^d}{\tau^{1/a_d}|\Sigma|^{1-1/a_d}}  \right\},
\end{equation}
then every signal supported on $E$ is uniquely determined among all signals with at most $|E|$ nonzero entries by its Fourier coefficients outside $\Sigma$.
\end{proposition}

\begin{proof}
Translation preserves additive energy, so we may assume that $b=0$. In the cases $d\equiv1\pmod4$ and $d\equiv3\pmod4$, $q\equiv1\pmod4$, Koh, Pham and
Vinh~\cite{KPV} prove
\[
 \E(A,A)\lesssim_d\frac{|A|^3}{q} +q^{(d-2)/2}|A|^2
 \qquad(A\subseteq S_j).
\]
Their statement for primitive radii extends to all nonsquare radii by dilation. In the remaining case, Koh and Pham~\cite{KP} prove
\[
 \E(A,A)\lesssim_d\frac{|A|^3}{q} +q^{(d-2)/2}|A|^2+q^{(d-3)/4}|A|^{5/2}.
\]
Thus, every $A\subseteq\Sigma$ satisfies ${\E(A,A)\lesssim_d B_{q,d}(|\Sigma|)|A|^2}$. Lemma~\ref{lem:energy-recovery} give the first condition in~\eqref{eq:recovery-earlier-energy}.

{For the second condition, let $t(S_j)$ be the subgroup parameter defined in~\cite[Section~4, (4.2)]{Shkredov}. Thus, $t(S_j)$ is the largest cardinality of an algebraic subgroup of the additive group $\F_q^d$ whose coset is contained in $S_j$. We have $t(S_j)=\tau$.} Indeed, suppose that $z+\Gamma\subseteq S_j$, where $\Gamma$ is an additive subgroup. Since $z\pm u\in S_j$ for $u\in\Gamma$, we have $Q(u)=z\cdot u=0$. Polarization then gives $u\cdot v=0$ for all $u,v\in\Gamma$. The $\F_q$-span of $\Gamma$ is therefore totally isotropic in $z^\perp$. By Lemma~\ref{lem:geometry}, its dimension is at most $h-1$, so $|\Gamma|\le\tau$. Equality is attained by an affine $(h-1)$-flat in the sphere.

We apply Lemma~1 in Section~3 of~\cite[p.~328]{Shkredov} with $k=d$, and equation~(4.5) of Theorem~1 in Section~4 of~\cite[p.~331]{Shkredov} with $\dim S_j=d-1$.
Since $t(S_j)=\tau$, these estimates give 
\[
 \E(A,A)^{a_d}|A|^{-(3\cdot2^d-4d-4)}  \le\norm{A}_{\mathcal U^d}\lesssim_d\tau|A|^d  \qquad(\varnothing\ne A\subseteq S_j).
\]
Here $\norm{A}_{\mathcal U^d}$ is the unnormalized Gowers quantity counting additive $d$-dimensional cubes in $A$. The estimates hold for sufficiently large $q$; the remaining field orders are covered by increasing the constant depending on $d$. Since $3a_d-1=3\cdot2^d-3d-4$, it follows that
\[
 \E(A,A)\lesssim_d\tau^{1/a_d}|A|^{3-1/a_d}.
\]
Apply~\eqref{eq:energy-power-recovery} with $\delta=1-1/a_d$ and $K=C_d\tau^{1/a_d}$ to obtain the second condition. Taking the smaller of the two constants completes the proof.
\end{proof}

\subsection{Comparison and sharpness}\label{subsec:recovery-comparison}

Throughout this subsection, write
\[
\tau=q^{(d-3)/2},\qquad a_d=2^d-d-1.
\]
We first give {an example} in every odd dimension $d\ge5$ where Theorem~\ref{thm:recovery} improves the sufficient conditions above. Suppose that $q$ is an odd square and $|\Sigma|=q^{(d-2)/2}$. Under the sign condition of Theorem~\ref{thm:recovery}, $j$ is then a nonsquare, and the Koh--Pham--Vinh case of Proposition~\ref{prop:recovery-earlier-energy} applies. The following table gives sufficient conditions for unique recovery among signals with at most $|E|$ nonzero entries. Here $c_d>0$ is chosen sufficiently small, depending only on $d$. The displayed thresholds are not asserted to be optimal. 
\begin{center}
\small
\renewcommand{\arraystretch}{1.35}
\begin{tabular}{lc}
\hline
Method & Sufficient condition on $|E|$\\
\hline
Classical uncertainty, \eqref{eq:recovery-classical}
 & $|E|<c_dq^{(d+2)/2}$\\
Previous extension estimate, \eqref{eq:recovery-previous}
 & $|E|<c_dq^{(d+2)/2}$\\
Stein--Tomas, \eqref{eq:recovery-stein-tomas}
 & $|E|<c_dq^{(d+1)/2}$\\
KPV energy through Lemma~\ref{lem:energy-recovery}
 & $|E|<c_dq^{(d+2)/2}/\log^2q$\\
Shkredov energy, \eqref{eq:energy-power-recovery}
 & $|E|<c_dq^{(d+2)/2+1/(2a_d)}$\\
Meshulam's uncertainty inequality
 & $|E|<c_dq^{(d+2)/2}$\\
Theorem~\ref{thm:recovery}
 & $|E|<c_dq^{(d+2)/2+1/(d+1)}$\\
\hline
\end{tabular}
\end{center}
The entries concerning energy follow by substitution into \eqref{eq:recovery-earlier-energy}. Since $1/(d+1)>1/(2a_d)$ for $d\ge5$, our theorem improves the power of $q$ in each of these sufficient conditions. This improvement does not depend on the logarithmic loss in the Koh--Pham--Vinh energy condition. 

The comparison with Meshulam~\cite[Theorem~1.2]{Meshulam} depends on the arithmetic of $q$. If $\rho_-<\rho_+$ are consecutive divisors of $q^d$ and $\rho_-\le |\Sigma|\le\rho_+$, his theorem gives 
\begin{equation}\label{eq:meshulam-comparison}
 |\operatorname{supp}v|  \ge\frac{q^d}{\rho_-\rho_+}(\rho_-+\rho_+-|\Sigma|)  \quad\text{if }v\ne0,\quad  \operatorname{supp}\widetilde v\subseteq\Sigma.
\end{equation}
The right-hand side decreases with $|\Sigma|$, so the assertion follows even when the Fourier support is a proper subset of $\Sigma$. Thus, uniqueness follows whenever $2|E|$ is smaller than the right-hand side of~\eqref{eq:meshulam-comparison}. For square $q$ and $|\Sigma|=q^{(d-2)/2}$, $|\Sigma|$ is a divisor of $q^d$,
and the right-hand side is $q^d/|\Sigma|=q^{(d+2)/2}$, as in the table.

More generally, let $q$ be an odd square and let $d\ge5$ be odd. Suppose that $|\Sigma|\approx q^\omega$, where $\omega$ is fixed and
\[
 \frac{d-2}{2}<\omega<\frac{d-2}{2}+\frac1{d-1}.
\]
Theorem~\ref{thm:recovery} gives a sufficient threshold of order
\[
 q^{\,d-\omega+(2\omega-d+3)/(d+1)},
\]
whose exponent exceeds $(d+2)/2$. The classical recovery threshold is $O(q^{(d+2)/2})$. The thresholds obtained from the previously known extension estimate, the Stein--Tomas estimate, and the KPV bound satisfy the same upper bound. The same holds for Meshulam's threshold by monotonicity and its value at $|\Sigma|=q^{(d-2)/2}$. The Shkredov threshold has exponent $d-\omega+(2\omega-d+3)/(2a_d)$, which is also smaller since $2\omega-d+3>0$ and $1/(d+1)>1/(2a_d)$. Thus, our threshold exceeds each of these earlier sufficient thresholds by a positive power of $q$ throughout this range of fixed exponents.

In contrast, let ${q}$ be prime and $d=2h+1\ge5$. For ${q^{h-1}\le |\Sigma|\le q^h}$,~\eqref{eq:meshulam-comparison} becomes
\[
 |\operatorname{supp}v|\ge q^2(q^{h-1}+q^h-|\Sigma|).
\]
At ${|\Sigma|\approx q^{h-1/2}}$ this is comparable to ${q^{h+2}}$. It therefore gives a stronger uniqueness condition than Theorem~\ref{thm:recovery} in the corresponding example over prime fields. More generally, for fixed $h-1<\omega<h$ and ${|\Sigma|\approx q^\omega}$, Meshulam's bound still permits ${|E|<c_dq^{h+2}}$, whereas~\eqref{eq:recovery-condition} permits a smaller power. The square-field assumption in the comparison above is therefore relevant.

For completeness, in dimension three there are more specific energy estimates. Let ${q}$ be prime and ${\eta_q(-j)=-1}$. Rudnev~\cite[Theorem~18, equation~(30)]{Rudnev} proves that
\[
 \E(A,A)\lesssim |A|k_0^2+|A|^{37/15}  \qquad(A\subseteq S_j,\ |A|<q^{15/11}),
\]
where $k_0$ is the maximum number of points of $A$ on an isotropic line. No such line is contained in $S_j$. Since the restriction of $Q-j$ to an isotropic line is affine linear, it has at most one zero, and hence $k_0\le1$. Applying~\eqref{eq:energy-power-recovery} with $\delta=7/15$ gives the sufficient condition
\[
 |E|<\frac{c q^3}{|\Sigma|^{7/15}} \qquad\text{when}\qquad |\Sigma|<q^{15/11}.
\]
This also shows why the recovery condition in Theorem~\ref{thm:recovery} is not asserted to be optimal in dimension three.

For every odd prime power $q$ and odd dimension $d=2h+1\ge3$, Theorem~\ref{thm:recovery} is sharp up to a constant when $|\Sigma|=\tau=q^{h-1}$. To see this, let
$\Sigma=z+U\subseteq S_j$ be an affine $(h-1)$-flat supplied by Lemma~\ref{lem:geometry}, and define $v$ by ${\widetilde v=\1_\Sigma}$. Fourier inversion gives
\[
 v(x)=q^{-d}|U|\chi(x\cdot z)\1_{U^\perp}(x),  \qquad |U^\perp|=q^{h+2}.
\]
The character sum vanishes outside $U^\perp$ because the nonzero $\F_q$-linear map $u\mapsto x\cdot u$ is onto $\F_q$. Partition $U^\perp=E_1\sqcup E_2$ into two sets whose cardinalities differ by one. The signals $g_1=v\1_{E_1}$ and $g_2=-v\1_{E_2}$ have the same observed Fourier coefficients, since $g_1-g_2=v$.
Each has at most $\lceil q^{h+2}/2\rceil$ nonzero entries. Thus, uniform uniqueness fails at this size, while \eqref{eq:recovery-condition} permits $|E|<\varepsilon_dq^{h+2}$. The classical condition~\eqref{eq:recovery-classical} already gives
the same order at $|\Sigma|=\tau$.

More generally, the ratio of the size permitted by \eqref{eq:recovery-condition} to the classical size $q^d/|\Sigma|$, apart from constants, is
\[
 \left(\frac{|\Sigma|}{\tau}\right)^{2/(d+1)}.
\]
Thus, classical uncertainty gives the larger size when $|\Sigma|<\tau$. At the other extreme, if $\Sigma=S_j$, our condition permits $|E|<c_dq^2$, whereas Stein--Tomas permits $|E|<c_dq^{(d+1)/2}$, which is larger for $d\ge5$. The affine-flat example establishes sharpness at $|\Sigma|=\tau$. It does not establish an optimal recovery threshold in the range where Theorem~\ref{thm:recovery} improves the earlier conditions.

\section*{Acknowledgements}

The second author was supported by the Research Startup Fund of ShanghaiTech University.

\section*{Declaration of competing interest}
The authors declare that they have no competing interests.

\section*{Data availability}
No data were used for the research described in the article.

\end{document}